\documentclass[a4paper]{IEEEtran}  % Comment this line out if you need a4paper

\IEEEoverridecommandlockouts                              % This command is only needed if 
\usepackage{cite}
\usepackage{float}

\usepackage{graphicx}
\usepackage{amssymb}
\usepackage[cmex 10]{amsmath}
\usepackage{array}
\usepackage{mdwmath}

\newtheorem{theorem}{Theorem}
\newtheorem{definition}{Definition}
\newtheorem{lemma}{Lemma}

\newtheorem{proposition}{Proposition}
\usepackage{subcaption}
\usepackage{epsfig}

\usepackage{times}

\usepackage{tikz}
\usetikzlibrary{shapes,arrows}

\usetikzlibrary{arrows,calc}

\tikzset{
        block/.style = {draw, rectangle, 
                        minimum height=1cm, 
                        minimum width=2cm},
        input/.style = {coordinate,node distance=1cm},
        output/.style = {coordinate,node distance=4cm},
        arrow/.style={draw, -latex,node distance=2cm},
        pinstyle/.style = {pin edge={latex-, black,node distance=2cm}},
        sum/.style = {draw, circle, node distance=1cm}
}

\title{Passivity   Analysis of Replicator Dynamics and its Variations}

\author{M. A. Mabrok 
\thanks{M. A. Mabrok (m.a.mabrok@gmail.com) With the School of Engineering, Australian College of Kuwait.}}

\begin{document}

\maketitle

\begin{keywords}
Learning in games, evolutionary games, passivity, population games.
\end{keywords}

\begin{abstract}
In this paper, we focus on studying the passivity properties of different versions of replicator dynamics (RD). RD is  an important class of evolutionary dynamics in  evolutionary game theory. Evolutionary dynamics describe how the population composition changes in response to the fitness levels, resulting in a closed-loop feedback system. RD is a deterministic monotone non-linear  dynamic that allows incorporation of  the distribution of  population types through  a fitness function. Here, in this paper, we use a  tools for control theory, in particular,  the  passivity theory, to study the stability  of the RD  when it is in action with evolutionary games. The passivity theory allows us to identify class of evolutionary games in which stability with RD is guaranteed.   We show that several variations of the first order RD   satisfy the standard loseless passivity property. In contrary, the second order RD do not satisfy the standard  passivity property, however, it satisfies a similar dissipativity property known as negative imaginary property. The  negative imaginary property of the second order RD allows us to identify the class of games that converge to a stable  equilibrium with the second order RD. 

\end{abstract}

\section{Introduction}

Population games  \cite{Sandholm2010, Hofbauer1998} is a class of games that  model interactions between  a large number of agents (named players in game theory contest), in which each player or agent's  payoff  depends on  his own  strategy and the distribution of strategies of other agents. There has been extensive research in a variety of settings, ranging from societal \cite{Young1998} to biological \cite{Smith1982} to engineered \cite{Marden2015}.

In population games, and  learning in games \cite{Fudenberg1998,Hart2005,Young2005}, one of the main points  is to  understand the  behavior of the agent  strategies on the long run.  In particular, understanding the convergence and divergence of population strategies to a particular solution concept such as Nash equilibrium. The  convergence or  divergence of such strategies will depend on both the game under consideration as well as the evolutionary dynamics \cite{Hofbauer2003,Hofbauer2009}.
Some specific games can exhibit inherent obstacles to convergence for a cluster of  evolutionary dynamics \cite{Hart2003}.

Replicator dynamics (RD) is one  of the most studied game dynamics in the literature of evolutionary game theory. RD was introduced as a non-linear first-order differential equations in \cite{taylor1978evolutionary} to model a single species   and given the name of RD in  \cite{Schuster1983}. The RD has been used extensively in modeling game dynamics including    biological systems  in order to predict the evolutionary  behavior without a detailed analysis of such biological factors as genetic or population size effects  \cite{cressman2014replicator}.

Passivity theory is a very well established  sub-class of control  theory, which implies  useful properties such as stability, and  the importance of passivity as tool in nonlinear control of interconnected systems---unlike Lyapunov stability criteria---relays  on the fact that any  set of  passive sub-systems in parallel or feedback configuration forms a passive  system. Passivity is an input-output property of a class of nonlinear physical systems which can only consume energy. The notion of passivity originally comes from electric circuits composed of resistors, capacitors, and inductors. The same definition
applies to analogous mechanical and hydraulic systems. This idea can be
extended to study electric circuits with nonlinear passive components and magnetic
couplings. For a detailed discussion on passive systems see [2, 3, 4, 5] and
the references therein, which show that passive systems theory received a great
deal of attention in the last few decades. When a passive system is connected to
a strictly passive system in a negative feedback loop, energy is strictly dissipated
as signals propagate around the loop, and hence the feedback interconnection is stable.   In other words,  by ensuring that every subsystem is passive, a complex structure of subsystems  can be built  to satisfy certain properties.

In \cite{Fox2013}, the relationship between passivity theory and class of population games and dynamics  was   established. It was shown that what is called contractive games can be considered as passive games where it satisfies similar   properties to passive systems. Similarly,  many  evolutionary dynamics also satisfy  the  passivity property. This implies that the feedback interconnections of passive evolutionary dynamics with passive game exhibit stable behavior. The above connection between  passivity theory and  population games and dynamics enables the opportunity to analyze in a similar way similar  classes  of  games and evolutionary dynamics, for instance, higher order games and higher order dynamics.

In the standard  population games, the fitness function of the  population strategies is a static function of the population composition. However, in higher order population games,  dependence can be dynamic, e.g., as a model of path dependencies \cite{Fox2013}. Similarly, in the standard form evolutionary dynamics, the number of states is equal to the number of population strategies. However, higher order dynamics is able to introduce different behaviors \cite{Hart2003,Sato2002,Shamma2005,Arslan2006}. Also, in  \cite{Laraki2013}, the higher order version of the replicator dynamics can  eliminate  weakly dominated strategies, which is not the case in the  canonical  replicator dynamics.

In this paper, we focus on studying the passivity property of several variations of   the well know replicator dynamics, including the higher order replicator dynamics. We show that the standard replicator dynamics satisfy the lossles passivity  property. Also, the local version of replicator dynamics satisfy the lossles  property. However, the higher order of the replicator dynamics do not satisfy the passivity property. Instead, it satisfy the negative imaginary property, which can be considered as a counterpart of the passivity theory. The Negative imaginary theory also have a similar stability conditions to those in passivity theory.

The remainder of this paper is organized as follows: Section II presents preliminary material on population games, passivity and negative imaginary systems. Section III presents the concept of higher order games and dynamics. Section IV  presents  the passivity property for different variations of the replicator dynamics. Section V inderduces the notion of passiviation for replicator dynamics. Finally, Section VI contains concluding remarks.

\section{Preliminaries and notations }\label{sec:PAN}

This section presents the required  preliminaries and notations  from input-output operators, game  theory, replicator dynamics  and passivity theory. 
\subsection{Input-output operators}
Dynamical system can be seen  as an operator defined on function spaces. Suppose that $ \mathbb{L}_2$ 
denote the Hilbert space of square integrable functions, which  maps  $ \mathbb{R}_+$ to $ \mathbb{R}^n$ with inner product  $\langle \cdot, \cdot \rangle$ and $\lVert \cdot \lVert$. Let $ \mathbb{L}_{2e}$ denote the  space of functions that are square integrable over finite intervals. 

\begin{align}\label{eq:1}
\mathbb{L}_{2e}=\left\lbrace f: \mathbb{R}_+ \rightarrow \mathbb{R}_{n}:\int_0^ \infty f(t)^Tf(t)dt \le \infty \; \forall \; T \in \mathbb{R} \right\rbrace.
\end{align}
For $ \mathbb{U}$,$ \mathbb{Y}$ $\subset$  $ \mathbb{L}_{2e}$ subsets of functions, an input-output operator is a mapping $S: \mathbb{U}\rightarrow \mathbb{Y}$.

\subsection{Passivity  theory}

Passivity theory is a useful tool to assess the stability of a feedback interconnection.
In brief, if both components of a feedback interconnection have the passivity property, then the closed loop interconnection is stable \cite{Lozano2013,Willems1972,Fox2013}. A game theoretic context  to passivity in games including the definition  of   \emph{passive games} was  given  in \cite{Fox2013}.

Define $\Sigma$ to be a   nonlinear dynamical system with the following state space description: 
\begin{align}\label{eq:nonlinearsys}
\dot{x} & =f(x,u), \;\;\;\;\;\; x(0)=x_0,\\
y & =g(x,u), \notag
\end{align}
where, $u(t) \in \mathbf{U} \subset \mathbb{R}^m$ is the system's input vector, $x_0 \in \mathbf{M} \subset \mathbb{R}^n$,  $y \in \mathbb{R}^m$ is the system's output vector and $x(t) \in \mathbb{R}^n$ is the system's state vector. Also, assume that for some classes of functions,
$\mathbb{U}$ and $\mathbb{Y}$, and for any initial condition in $\mathbf{M}$, there exists a solution for all $u \in \mathbb{U}$ resulting in $y \in \mathbb{Y}$ and
$x(t) \in \mathbf{M}$ all $t \geq 0$. Next, we present  two definitions  for  passive system from both state space and input-output perspectives.

\begin{definition}
The nonlinear  system $\Sigma$ with state space \eqref{eq:nonlinearsys} is said to be passive if there exists a  \textit{storage function } $ L:\mathbf{M}\rightarrow \mathbb{R_+}$ such that for all  $x_0 \in \mathbf{M}$ and $t \geq 0$,
\begin{align}\label{eq:storage1}
L(x(t))\leq L(x_0)+\int_0^t\!  u(\tau)^T y(\tau) d \tau.
\end{align}
\end{definition}

The input-output definition of passivity property is given as follows: 
\begin{definition}
The input-output operator $S: \mathbb{U}\rightarrow \mathbb{Y}$ is said to be passive    if there exist  constant $\alpha$  such that:
\begin{align}\label{eq:inerpassive}
\langle Su, u\rangle_T \geq\alpha, \forall \; u \in \mathbb{U}, T\in \mathbb{R}_+,
\end{align}
and input strictly passive if there exist $\beta > 0$  and $\alpha$ such that
\begin{align}\label{eq:inerpassive2}
\langle Su, u\rangle_T \geq\alpha, +\beta\langle u, u\rangle_T \forall\; u \in \mathbb{U}, T\in \mathbb{R}_+,
\end{align}
where, $\langle f, g \rangle_T=\int_0^Tf(t)^Tg(t)dt$.
\end{definition}

In the case of  equality in the inequalities  \eqref{eq:storage1} and \eqref{eq:inerpassive}, the system is said to be \textit{lossless}. 

The stability of the  feedback interconnection between passive systems is  a fundamental result in passivity theory (e.g., \cite{VanderSchaft2012}).  That is, the negative feedback interconnection between a passive system  $\Sigma_1$ and strictly passive  $\Sigma_2$,  as shown in  Fig. \ref{fig1},  is stable  feedback interconnection. Also, the closed loop system from $r$ to $y_1$ is passive. 

\begin{figure}
\begin{center}
\begin{tikzpicture}[auto, node distance=2.5cm,>=latex']
    % We start by placing the blocks
    \node [input, name=input] {};
    \node [sum, right of=input] (sum) {};
   \node [block, right of=sum] (system) {$\Sigma_1$};
    \node [output, right of=system] (output) {};
    \node [block, below of=system] (controller) {$\Sigma_2$};
    \draw [draw,->] (input) -- node {$r$} (sum);
    \draw [->] (sum) -- node {$u_1$} (system);
    \draw [->] (system) -- node [name=y] {$y_1$}(output);
    \draw [->] (y) |- node [above,pos=0.79]{$u_2$} (controller) ;
    \draw [->] (controller) -| node[pos=0.99] {$-$} 
        node [near end] {$y_2$} (sum);
\end{tikzpicture}
\end{center}
        \caption{Negative  feedback interconnection. }\label{fig1}
\end{figure}
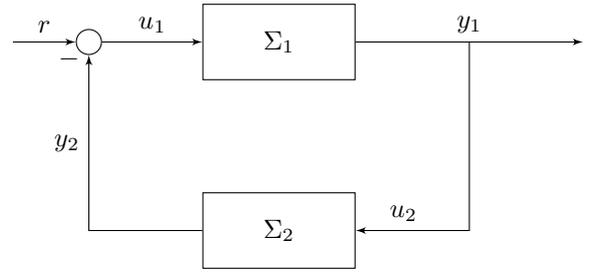

\subsection{Negative imaginary systems theory }

Negative imaginary (NI) systems theory was introduced by  Lanzon  and  Petersen  in \cite{lanzon2008,petersen2010} for linear time invariant systems to control systems with flexible structure dynamics. Several  generalization for the the  negative imaginary  theory can be found in \cite{xiong21010jor, Mabrok2014c, Mabrok2015}.  NI systems  theory complements the limitations of passivity theory in linear systems,  which is only applicable for systems with relative degree zero or one. Recently, nonlinear version of negative imaginary system theory was developed in \cite{Ghallab2018}. 
In this section, we recall  the     definition of NI and
SNI systems as given in \cite{Mabrok2013}. We also define some 
notations that will be used  in the  paper.

Consider the following LTI system,
\begin{align}
\label{eq:xdotn}
&\dot{x}(t) = A x(t)+B u(t), \\
\label{eq:yn} &y(t) = C x(t)+D u(t),
\end{align}%
where $A \in \mathbb{R}^{n \times n},B \in \mathbb{R}^{n \times m},C
\in \mathbb{R}^{m \times n},$ $D \in \mathbb{R}^{m \times m},$ and
with the  square transfer function matrix $G(s)=C(sI-A)^{-1} B+D$.
The transfer function matrix $G(s)$ is said to be strictly proper if
$G(\infty)=D=0$. We will use the notation $
\begin{bmatrix}
\begin{array}{c|c}
A & B \\ \hline C & D
\end{array}
\end{bmatrix}$ to denote the state space realization
\eqref{eq:xdotn}, \eqref{eq:yn}.

The NI system is defined as follows; 
\begin{definition}\cite{Mabrok2012,Mabrok2013}\label{Def:NI}
A square transfer function matrix $G(s)$ is NI if  the following
conditions are satisfied:
\begin{enumerate}
\item $G(s)$ has no pole in $Re[s]>0$.
\item For all $\omega >0$ such that $s=j\omega$ is not a pole of $G(s
)$,
\begin{equation}\label{eq:NI:def}
    j\left( G(j\omega )-G(j\omega )^{\ast }\right) \geq 0.
\end{equation}
\item If $s=j\omega _{0}$ with $\omega _{0}>0$ is a pole of $G(s)$, then it is a simple pole and the residue matrix $K=\underset{%
s\longrightarrow j\omega _{0}}{\lim }(s-j\omega _{0})jG(s)$ is
Hermitian and  positive semidefinite.
 \item If $s=0$ is a pole of $G(s)$, then
$\underset{s\longrightarrow 0}{\lim }s^{k}G(s)=0$ for all $k\geq3$
and $\underset{s\longrightarrow 0}{\lim }s^{2}G(s)$ is Hermitian and
positive semidefinite.
\end{enumerate}
\end{definition}
%\end{definition}

\begin{definition}\cite{xiong21010jor}
A square transfer function matrix $G(s)$ is SNI if  the
following conditions are satisfied:
\begin{enumerate}
\item ${G}(s)$ has no pole in $Re[s]\geq0$.
\item For all $\omega >0$, $j\left( {G}(j\omega )-{G}(j\omega )^{\ast }\right) > 0$.
\end{enumerate}
\end{definition}

Here, we present an NI lemma.
\begin{lemma}\label{NI-lemma}
 Let $
\begin{bmatrix}
\begin{array}{c|c}
A & B \\ \hline C & D
\end{array}
\end{bmatrix}$ defining the system (\ref{eq:xdotn})-(\ref{eq:yn}) be a minimal realization of the transfer function matrix  $%
G(s)$. Then, $G(s)$ is NI if and only if $D=D^T$ and there exist
matrices $P=P^{T}\geq 0$,
 $W\in \mathbb{R}^{m \times m}$, and $L\in \mathbb{R}^{m \times n}$
 such that the following linear matrix inequality (LMI) is
satisfied:
\begin{small}
\begin{align}\label{LMI:PR}
\begin{bmatrix}
PA+A^{T}P & PB-A^{T}C^{T} \\
B^{T}P-CA & -(CB+B^{T}C^{T})%
\end{bmatrix}%
 =
\begin{bmatrix}
-L^{T}L & -L^{T}W \\
-W^{T}L & -W^{T}W%
\end{bmatrix}%
\leq 0.
\end{align}
\end{small}
\end{lemma}

The positive feedback interconnection between an NI
system  with transfer function matrix  $G(s)$ and an SNI system with
transfer function matrix $\bar{G}(s)$ as shown in Fig.
\ref{conn:NI:SNI}. Also, suppose that the transfer function matrix
$G(s)$ has a minimal state space realization  $
\begin{bmatrix}
\begin{array}{c|c}
A & B \\ \hline C & D
\end{array}
\end{bmatrix},$ and $\bar{G}(s)$ has a minimal state space realization  $
\begin{bmatrix}
\begin{array}{c|c}
\bar{A} & \bar{B} \\ \hline \bar{C} & \bar{D}
\end{array}
\end{bmatrix}.$ Furthermore, it is assumed that the matrix $I-D\bar{D}$   is nonsingular. Then the closed system has a system matrix given by

\begin{figure}
  \centering\includegraphics[width=7cm]{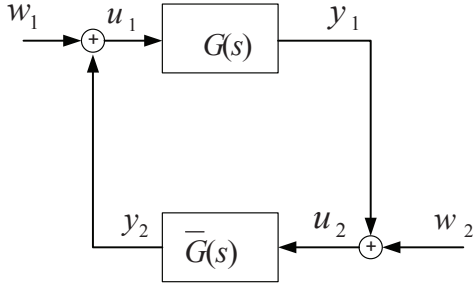}\\
  \caption{A negative-imaginary feedback control system. If the plant
transfer function matrix $G(s)$ is NI and the controller transfer
function matrix $\bar{G}(s)$ is SNI, then the positive-feedback
interconnection is internally stable if and only if the DC gain
condition, $\lambda_{max}(G(0)\bar{G}(0))<1,$ is satisfied.
}\label{conn:NI:SNI}
\end{figure}

\begin{small}
\begin{align}
\breve{A} =\begin{bmatrix} A+B\bar{D}(I-D\bar{D})^{-1}C &
B\bar{C}+B\bar{D}(I-D\bar{D})^{-1}D\bar{C} \\
\bar{B}(I-D\bar{D})^{-1}C & \bar{A}+\bar{B}(I-D\bar{D})^{-1}D\bar{C}%
\end{bmatrix}.%
\label{mat:A:CL}%\\
%\breve{B} =&%
%\begin{bmatrix}
%B\bar{D} \\
%\bar{B}%
%\end{bmatrix}%
%(I-D\bar{D})^{-1}, \nonumber\\
%\breve{C} =&(I-D\bar{D})^{-1}%
%\begin{bmatrix}
%C & D\bar{C}%
%\end{bmatrix},\nonumber
%\\ \nonumber
%\breve{D} =&(I-D\bar{D})^{-1}.\nonumber
\end{align}
\end{small}%
Moreover, the positive feedback interconnection between $G(s)$ and
$\bar{G}(s)$ as shown in Fig. \ref{conn:NI:SNI} and denoted
$[G(s),\bar{G}(s)]$ is said to be  internally stable if  the
closed-loop system  matrix $\breve{A}$ in \eqref{mat:A:CL} is
Hurwitz; e.g., see \cite{Glover1996}.

A nonlinear definition  of NI system is given as follows: 
 Consider the following general nonlinear system
\begin{align}
\label{x}\dot{x}&= f(x,u)\\
\label{y} y&= h(x)
\end{align}
where $f:\mathbb{R}^n\times \mathbb{R}\rightarrow\mathbb{R}^n$ is Lipschitz continuous function and  $h:\mathbb{R}^n\rightarrow\mathbb{R}$ is a class $C^1$ function.
\vspace{.25cm}

\begin{definition}\label{def:nonNI}
The nonlinear system \eqref{x}, \eqref{y} is Negative Imaginary if there exists a positive storage function $V:\mathbb{R}^n\rightarrow\mathbb{R}$ of a class $C^1$ such that
\begin{equation}\label{lyapd}
    \dot{V}(x(t))\leq \dot{y}(t)u(t).
\end{equation}
\end{definition}

\begin{definition}\label{def:nonNI2}
The nonlinear system \eqref{x}, \eqref{y} is Negative Imaginary if there exists a positive storage function $V:\mathbb{R}^n\rightarrow\mathbb{R}$ of a class $C^1$ such that
\begin{equation}\label{lyapd2}
V(x(t)) \leq V(x(0)) +\int_0^ty(s)^T\int_0^su(\tau)d\tau ds,
\end{equation}
for all $t > 0$. 

Or,
\begin{equation}\label{lyapd3}
\dot{V}(x(t)) \leq y(t)^T\int_0^t u(\tau)d\tau.
\end{equation}
\end{definition}

The main difference between the  NI definition given in \eqref{def:nonNI} and  passivity definition is the use of $\dot{y}(t)$ instead of ${y}(t)$ in calculating the supplied energy rate.

\subsection{Population games}

A game, in general, consist of three basic elements.  Number of \textit{players} $N$:   are  the  decision makers in  the game context. 
\textit{Strategies} $S$:  are  the set   of actions that a particular player will play  given a set of conditions and circumstances that will emerge in  the game being played. 
\textit{Payoff} $P$: is the reward which a player receives from playing  at a particular strategy. In general, game is viewed as a mapping from \textit{strategies} $x\in X$, where $X=\{x: \;\; \Sigma_{i \in S}x_i=1\}$, to \textit{Payoff} $F(x)$. This viewpoint can be extended in dynamic setting to a mapping of strategy trajectories $x(\cdot)$ to payoff trajectories $F(\cdot)$. In other words,  games can be described as  a dynamical systems with states, inputs and outputs.

 A single population game has a set of  strategies $S =\{1, 2, ..., m\}$ and  a set of strategy distributions 
$X=\{x: \;\; \Sigma_{i \in S}x_i=1\}$. Since strategies lie in the simplex, admissible changes in strategy are restricted to the tangent space $TX=\{z\in  \mathbb{R}^m : \sum_{i\in s} z_i=0\}$.  

Now, define $F(x) : S\longrightarrow \mathbb{R}^m$ to be  the payoff function that  associate each strategy distribution in $S$ with a payoff vector so that $F_i: S\longrightarrow \mathbb{R}^m$ is the payoff to strategy $i\in S$. Also, define $DF(x)$ to be the Jacobian matrix of $F(x)$.  A state $x^*\in X$ is a \textit{Nash equilibrium} if  each strategy in the support of $x$ receives the maximum payoff available to the population.

\subsection{Evolutionary dynamics}

In general game theory setup, the prediction of the game outcome and its equilibrium analysis depends on  what game players know about it. In such setup, a assumption such as full rationality of players is always made. 

A slightly different approach to study the  game behavior, in particular  for games that have repetitive interactions among a large number of players, is disequilibrium analysis. Here, players meant to adjust and revise the chose of their strategies based on the their current strategies as well as the state of the collective  population when a feedback is given. In this setup, the \textit{revision protocol}, is used to organize 
the procedure of how and when agents change their strategies.

\begin{figure}
\begin{center}
\begin{tikzpicture}[auto, node distance=3cm,>=latex']
    % We start by placing the blocks
    \node [input, name=input] {};
    \node [sum, right of=input] (sum) {};
   \node [block, right of=sum] (system) {Dynamics};
    \node [output, right of=system] (output) {};
    \node [block, below of=system] (controller) {  Game};
    %\draw [draw,->] (input) -- node {$r$} (sum);
    \draw [->] (sum) -- node {$P$} (system);
    \draw [->] (system) -- node [name=y] {$X$}(output);
    \draw [->] (y) |- node [above,pos=0.79]{$X$} (controller) ;
    \draw [->] (controller) -| node[pos=0.99] {$+$} 
        node [near end] {$P$} (sum);
\end{tikzpicture}
\end{center}
        \caption{Feedback interconnection of a dynamic and a game, where $P$ is the payoffs and $X$ is the strategies. }\label{fig:dy:game}
\end{figure}
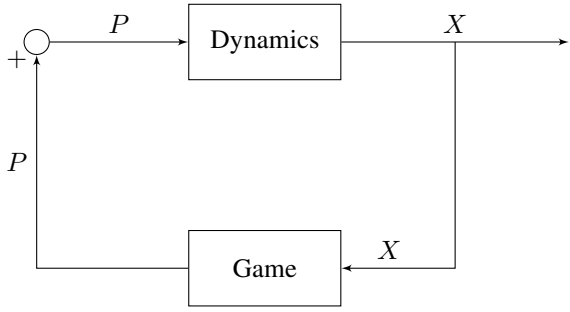

Revision protocol can be considered as a mapping from payoffs to strategy distributions under a particular  conditional switch rate $\rho_{ij}$. A population game $F$ with a revision protocol $\rho_{ij}$ creates the \textit{mean dynamic}, which is given as follows:

\begin{align}\label{eq:mean}
\dot{x_i}=\sum_{j \in S} x_j\rho_{ij}-x_i\sum_{j \in S} \rho_{ij}. 
\end{align}

Throughout the mean dynamic given in \eqref{eq:mean} and different revision protocols $\rho_{ij}$, one can create different evolutionary dynamics that maps payoffs to new strategy distribution. Then, this new  strategy distribution is mapped by the game to generate new vector of payoffs  as shown in Fig. \ref{fig:dy:game}. This analogy  is similar to the feedback interconnections in control systems setups. 

For instance, the replicator dynamic is generated  using the  imitation of success  revision protocol, which is given as follows:
\begin{align}\label{eq:RP1}
\rho_{ij}= x_j(p_j -K),
\end{align}
where $K$ is a constant less than or equal to any $p_j$. By substituting \eqref{eq:RP1} into the mean dynamic \eqref{eq:mean}, we get the replicator dynamics equations as follows: 

\begin{align}\label{eq:RD}
\dot{x_i}=x_i(p_i-\sum_j x_j p_j),
\end{align}
where $p_i$ is the payoff for using strategy $i$.

Another representation of the replicator dynamics Equation  \eqref{eq:RD} is  cascaded connection between an integrator and Gibbs distribution as shown in Fig. \ref{fig:RD-nonp}. 
\begin{figure}[H]
%\begin{center}

\centering{\includegraphics[height=3.8 cm]{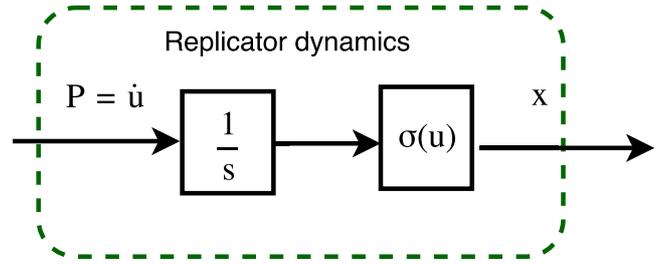}}
\caption{Replicator dynamics  \eqref{eq:RD}  as a cascaded connection between  integrator and Gibbs distribution.  }\label{fig:RD-nonp}.
%\end{center}
\end{figure} 

\section{ Higher-order dynamics and games}
As we indicated in the previous section, the standard   game is a static mapping from strategies $X$  to a set of real valued payoffs $P$,
\begin{equation*}
P=F(x),
\end{equation*}
and the evolutionary dynamic is restricted first order mapping from payoffs $P$ to strategies $X$,
\begin{equation*}
\dot{x}=V(x,p).
\end{equation*}
The dynamical view of this feedback loop can be extended to a mapping of strategy trajectories to payoff trajectories. This viewpoint allows the introduction of  generalized  forms of dynamics and games, such as higher-order  dynamics and games, to generate these trajectories. 

Higher-order  dynamics can be introduced---independent of the game---through  auxiliary states to the the first order dynamics \cite{Shamma2005,Arslan2006}, which can be interpreted as path dependency. Also, similar  higher-dynamics can be obtained by the direct derivative of the first order dynamics \cite{Laraki2013}.  It has been shown in \cite{Hart2003,Sato2002,Shamma2005,Arslan2006,Laraki2013} that   modification of the standard dynamics   can exhibit qualitatively different behaviors. One form of  generalized higher-order dynamics obtained by inducing  an auxiliary state $z$ int the dynamics. In this case, a higher-order dynamics takes the  following form:
\begin{equation*}
\begin{aligned}
\dot{z}&=f(z,p)  \;\;\;\;\; z(0)=z_0,\\
x&=g(z,p),
\end{aligned}
\end{equation*}
where, for all $x(0)$ in the simplex and for all bounded $p(\cdot)$,  there exists $x^*$ and $p^*$ such that if $p(t)\rightarrow p^*$, then, $x(t)\rightarrow x^*$.   

Furthermore,  higher-order  dynamics can be introduced by  \textit{processing} the  payoff. Here, the processing might  include  an integral  action, which means taking the payoff history into the consideration while generating the new strategies. Also, processing the payoff can be introduced thought a derivative action or anticipating the payoff trend, i.e.,   taking  into  consideration the future of the payoff or the trend that the payoff is following  \cite{Fox2013,Arslan2006}.

 Similarly,  static games  can be generalized by introducing  internal dynamics into the game. This concept is illustrated in \cite{Fox2013} through dynamically modified payoff function coupled with the static game. Therefore, we view  the  higher-order games as a generalization of   standard games  by introducing internal dynamics into the game, i.e., \textit{dynamical system} mapping from strategies $X$  to payoffs $P$. 
 
  Auxiliary state $z$ can be used to induce dynamics into static  games to generate  higher-order games as follows; 
  
   \begin{equation*}
   \begin{aligned}
   \dot{z}&=f(z,x) \;\;\;\;\; z(0)=z_0,\\
   p&=g(z,x),
   \end{aligned}
   \end{equation*}
   such that for all $x(0)$ in the simplex, and for all bounded $p(\cdot)$,  there exists $x^*$ and $p^*$ such that if $x(t)\rightarrow x^*$, there exists a unique $F(x^*)=\lim\limits_{t\rightarrow \infty}P(t)=p^*$.  This implies that each higher-order game converge and associated a stander static game at the steady state.  
   
   %also, $x(t)$ stays in the simplex and

Higher-order games can also be viewed as a \emph{processing or filtering}  the strategies  before computing its payoff. 
For instance, inducing a low pass filter that fillers out the  high frequency strategies, will transform  the game to a  dynamical system that  maps  strategies $X$  to payoffs $P$.  

Path dependency  is another motivation of  higher-order models in games. Here, path dependency means how a particular set of  choices  is affected   by the choices that  has been  made in the past, even if the  past circumstances may no longer be relevant \cite{Arrow1963}. 

%This implies that path dependency  can refer either to outcomes at a single moment in time, or to long-run equilibria of a process. In other words, path dependency means that history matters.

\section{Passivity Analysis for Different Variations of the Replicator Dynamic}\label{sec:HOG}
This section presents  the main contribution of the paper.

The definition of stability in this context implies that there
is an evolutionary stable state i.e., rest point, where the
distance between the population distribution and this rest point
decreases along the population trajectories, i.e., the population
converge to this stable state. Therefore, unstable feedback loop
between learning rule and game means that the feedback will
not converge to a rest point.

\subsection{Standard Replicator dynamics}\label{sec:RD}
Replicator dynamics are an important class of evolutionary dynamics that originated from models of natural evolution \cite{Schuster1983}. They provide a way to represent selection among a population of diverse types.

 We first show that first order replicator dynamics are indeed passive dynamics. In particular, replicator dynamics belongs to special class of passive systems known as lossless systems.

 The standard replicator dynamics is given in \eqref{eq:RD}:

Let $x^*$  be  a Nash equilibrium for the replicator dynamics. Define  $e_{x_i}=x_i-x^*_i$ to be the deviation from the equilibrium.   The following theorem shows that the first order replicator dynamics from the payoff $p_i$ to the error $e_{x_i}$ belongs to  passive (lossless) systems. 
\begin{theorem}
 The replicator dynamics given in \eqref{eq:RD} are passive (lossless) mapping from payoffs $p$ the error $e$.
\end{theorem}
\begin{proof}
Using  $e_{x_i}=x_i-x^*_i$, the replicator dynamic equation \eqref{eq:RD} can be written as follows:
\begin{align}
\dot{e}_{x_i}=(e_{x_i}+x^*_i)(p_i-\sum_j (e_{x_j}+x^*_j) p_j).
\end{align}
Define the following storage function, 
\begin{align}
V(e_{x})=-\sum_i x^*_i\ln\frac{e_{x_i}+x^*_i}{x^*_i}. \label{eq:storageFun}
\end{align}
Note  that $V(0)=0$, and 
\begin{align*}
V(e_{x})&=-\sum_i x^*_i\ln\frac{e_{x_i}+x^*_i}{x^*_i}\\&\geq -\ln\sum_i x^*_i\frac{e_{x_i}+x^*_i}{x^*_i}\\&=\ln(\sum_i x_i).
\end{align*}
However, $\sum_i x_i=1$. It follows that $V(e_{x})\geq 0$.

Now, the derivative of the storage function is given as follows:
\begin{align*}
\dot{V}(e_{x})&=-\sum_i x^*_i\frac{1}{e_{x_i}+x^*_i}\dot{e}_{x_i}\\
&=-\sum_i x^*_i(p_i-\sum_j (e_{x_j}+x^*_j) p_j)\\
&=-\sum_i x^*_ip_i+\sum_i x^*_i(\sum_j e_{x_j}p_j+\sum_jx^*_j p_j)\\
&=-\sum_i x^*_ip_i+\sum_j e_{x_j}p_j+\sum_jx^*_j p_j\\
&=\sum_j e_{x_j}p_j.
\end{align*}
This  implies that replicator dynamics are passive (lossless) system.
\end{proof}

To illustrate the passivity propriety of the standard replicator dynamics, consider the feedback interconnection between  the standard replicator dynamics and the well known   rock paper scissors game.  The payoff function is given as follows;
\begin{equation}\label{eq:rpsgame}
P=Ax,
\end{equation}
where 
\begin{equation*}
  A = \begin{pmatrix}
  0 & -l &  \omega \\
  \omega  & 0 &  -l \\
 -l &  \omega &  0
  \end{pmatrix}.
\end{equation*}
Here, $l,\omega$ are positive numbers. Since the  rock paper scissors game is a static game, i.e., there is  no dynamics in the mapping from $X$ to $P$. In other words, the rock paper scissors game represents a memoryless system. This game is passive if and only if $X^T P \geq 0$. The follwing shows   under which  conditions the rock paper scissors game satisfies the  passivity property. Consider a vector $X^T= (x_1, x_2, x_3)$. Using \eqref{eq:rpsgame}, it follows that;
\begin{align*}
  X^T\; P&= X^T A X  \\
  & = \begin{pmatrix}
  x_1& x_2 & x_3
  \end{pmatrix} \begin{pmatrix}
  0 & -l &  \omega \\
  \omega  & 0 &  -l \\
 -l &  \omega &  0
  \end{pmatrix} \begin{pmatrix}
  x_1\\ x_2 \\ x_3
  \end{pmatrix}\\
  &=(w-l)(x_1 x_2 + x_2 x_3 +x_1 x_3).
\end{align*}
Since the vector $X$ is a probability  distribution, this implies that the quantity $(x_1 x_2 + x_2 x_3 +x_1 x_3)$ is a non-negative quantity. It follows that the sign of  $ X^T\; P$  only depends on the sign of $(w-l)$.  This implies that  the rock paper scissors game  is strictly passive if $\omega >l$, and non-passive if $\omega<l$. Moreover, in the standard case, i.e.,  $\omega=l$ the game is lossless. 

Fig. \ref{RPS_RD_all} demonstrates the three different cases of the feedback interconnection between the rock paper scissors game  \eqref{eq:rpsgame}  and the standard  replicator dynamics \eqref{eq:RD}. The top figure represents the case where $\omega >l$, i.e., strictly passive game, the middle figure represents the case where $\omega =l$, i.e., lossless passive game, and the bottom  figure represents the case where $\omega <l$, i.e., non-passive game.

\begin{figure}[H]
%\begin{center}
\centering{\includegraphics[height=5.5 cm]{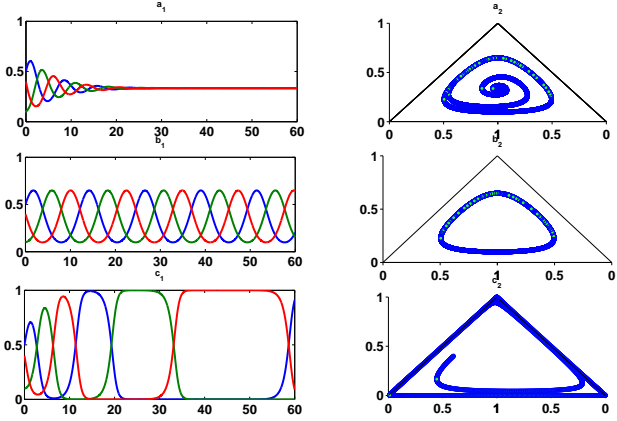}}
\caption{The feedback interconnection between the three different scenarios of rock paper scissors game  \eqref{eq:rpsgame}  and the standard  replicator dynamics \eqref{eq:RD} projected into the simplex.  }\label{RPS_RD_all}.
%\end{center}
\end{figure}

\subsection{Distributed  Replicator Dynamics}\label{sec:RD}

 Generally, the stranded Replicator dynamics   describes
the  dynamics of  the deterministic limit
of an infinitely large, well-mixed population, where the  spatial effects are not considered. Here, 
‘Well-mixed’ indicate  that there is no particular structure imposed into the game. In other words,  all agents  are equally likely to interact with one  another. This can be represented as a fully connected graph (complete graph) \cite{ohtsuki2006replicator}.

However, in structured populations (non-well mixed populations), the individual agents  set on  the vertices of a graph and the edges
of  determine which agents interact with each other \cite{ohtsuki2006replicator}. This scenario is called  distributed Replicator dynamics in \cite{bravo2015distributed}. Also, it is called local Replicator dynamics in \cite{pantoja2011dispatch}. The distributed (or local) Replicator dynamics is represented as follows:

\begin{align}\label{eq:DRDS}
\dot{x_i}=x_i(p_i\sum_{j\in N_i} x_j-\sum_{j\in N_i} x_j p_j),
\end{align}
meant to  models  the inference of population dynamics that involve  non-well mixed populations.

In the case where the payoff $P =Ax$, the stander replicator dynamics with can be written as follows;
\begin{align}\label{eq:DRD}
\dot{x_i}=x_i\left((Ax)_i-x^T Ax\right),
\end{align}

In \cite{hilbe2011local} evolutionary models for infinite and finite populations: While the population itself is infinite, interactions and reproduction occurs in random groups of size $N$. This leads to a modified payoff matrix in the following form:

\begin{align}\label{eq:ADRD}
A_{new} = A - \frac{A+A^T}{N}
\end{align}
Therefore, the local replicator dynamics is given as follows; 

\begin{align}\label{eq:LRD}
\dot{x_i}=x_i\left((A_{new}x)_i-x^T A_{new}\;x\right),
\end{align}

The following proposition shows that the local replicator dynamics is also lossless passive from the error to the payoff.
\begin{proposition}
The local replicator dynamics defined in \eqref{eq:LRD} from the error $e_{x_i}$ to the payoff $P=Ax$ is lossless passive. 
\end{proposition}

\begin{proof}
Using  $e_{x_i}=x_i-x^*_i$, the replicator dynamic equation \eqref{eq:LRD} can be written as follows:
\begin{align}
\dot{e}_{x_i}=&(e_{x_i}+x^*_i)[(A_{new}(e_{x}+x^*))_i\\
&-(e_{x}+x^*)^T \;A_{new}\; (e_{x}+x^*)].
\end{align}
Define the following storage function, 
\begin{align}
V(e_{x})=-\left(\frac{N}{N-2}\right) \sum_i x^*_i\ln\frac{e_{x_i}+x^*_i}{x^*_i}.
\end{align}
Note  that $V(0)=0$, and 
\begin{align*}
V(e_{x})&=-\left(\frac{N}{N-2}\right)\sum_i x^*_i\ln\frac{e_{x_i}+x^*_i}{x^*_i}\\&\geq -\left(\frac{N}{N-2}\right)\ln\sum_i x^*_i\frac{e_{x_i}+x^*_i}{x^*_i}\\&=\left(\frac{N}{N-2}\right) \ln(\sum_i x_i),
\end{align*}
However, $\sum_i x_i=1$. It follows that $V(e_{x})\geq 0$.

Now, the derivative of the storage function is given as follows:
\begin{align*}
&\dot{V}(e_{x})\\&=-\left(\frac{N}{N-2}\right)\sum_i x^*_i\frac{1}{e_{x_i}+x^*_i}\dot{e}_{x_i}\\
&=-\left(\frac{N}{N-2}\right)\sum_i x^*_i[(A_{new}(e_{x}+x^*))_i\\&\;\;\;\;\;\;\;\;\;\;\;\;\;-(e_{x}+x^*)^T \;A_{new}\; (e_{x}+x^*)]\\
&=-\left(\frac{N}{N-2}\right) [x^{*T}A_{new}(e_{x}+x^*)\\&\;\;\;\;\;\;\;\;\;\;\;\;\;-(e_{x}+x^*)^T \;A_{new}\; (e_{x}+x^*)]\\
&=-\left(\frac{N}{N-2}\right) [x^{*T}A_{new}e_{x}+x^{*T}A_{new}x^*-e_{x}^{T}A_{new}e_{x}\\&\;\;\;\;\;\;\;\;-e_{x}^T \;A_{new}\; x^*- x^{*T} \;A_{new}\; e_{x} -  x^{*T}A_{new}x^*         ]\\
&=-\left(\frac{N}{N-2}\right) [-e_{x}^{T}A_{new}e_{x}-e_{x}^T \;A_{new}\; x^*]\\
&=\left(\frac{N}{N-2}\right) e_{x}^{T}A_{new}(e_{x}+ x^*)\\
&=\left(\frac{N}{N-2}\right) e_{x}^{T}\left(A - \frac{A+A^T}{N}\right)x\\
&=\left(\frac{N}{N-2}\right) \left(e_{x}^{T}Ax - \frac{1}{N}(e_{x}^{T}Ax+e_{x}^{T}A^Tx)\right)\\
&=\left(\frac{N}{N-2}\right) \left(e_{x}^{T}Ax - \frac{2}{N}e_{x}^{T}Ax\right)\\
&=\left(\frac{N}{N-2}\right) \left(1- \frac{2}{N}\right)e_{x}^{T}Ax\\
&=e_{x}^{T}Ax\\
&=e_{x}^{T}P
\end{align*}
This  implies that replicator dynamics are passive (lossless) system.
\end{proof}

\subsection{Second order replicator  dynamics }

Second order dynamics can be introduced into the evolutionary  dynamics through several forms \cite{laraki2013higher,Fox2013}. In this section, we follow the approach  in \cite{Fox2013}, which use an auxiliary state to introduce the Second order dynamics to derive two versions of the Second order replicator dynamics.

\subsubsection{Integral action Second order replicator dynamics}

One form of the second order replicator dynamics can be obtained by introducing an auxiliary state $\widehat{p}$ in the payoff function. This results in the following dynamics:
\begin{subequations}\label{eq:RD2order}
\begin{align}
\dot{x_i}&=x_i(\widehat{p}_i-\sum_j x_j \widehat{p}_j)\label{eq:RD2ordera}\\
\dot{\widehat{p}_i}&=p_i.\label{eq:RD2orderb}
\end{align}
\end{subequations}
The  equilibrium conditions are $ p_i^*=0$, $x_i^*=\sigma_{\text{max}}(\widehat{p}_i^*)$, and $\widehat{p}_i^*=1$ .

The Equation \eqref{eq:RD2orderb} can be written as follows; 

\begin{align}
\dot{\widehat{p}_i}&=p_i\\ 
\widehat{p}_i &= \int_0^t p_i(\tau) d \tau.
\end{align}
It has been shown in \cite{mabrok2016a} that the second order replicator  dynamics is not passive. In particular, it has been shown that the  linearization of the replicator dynamics \eqref{eq:RD2order} is reduced to a double integrator, which is not passive.

The following proposition proves that the second order replicator defined in \eqref{eq:RD2order} is a negative imaginary system according to the definition given in \eqref{def:nonNI2}. 

\begin{proposition}\label{HoRDNI}
The second order  replicator dynamics defined in \eqref{eq:RD2order} from the error $e_{x_i}$ to the payoff $P$ is a negative imaginary system. 
\end{proposition}

\begin{proof}
Using  $e_{x_i}=x_i-x^*_i$, the replicator dynamic equation \eqref{eq:RD2order} can be written as follows:
\begin{align}
\dot{e}_{x_i}&=(e_{x_i}+x^*_i)(\widehat{p_i}-\sum_j (e_{x_j}+x^*_j) \widehat{p_j})\\
 \widehat{p}_i &= \int_0^t p_i(\tau) d \tau.
\end{align}
Using the  storage function \eqref{eq:storageFun}, it follows that;

\begin{align*}
\dot{V}(e_{x})&=-\sum_i x^*_i\frac{1}{e_{x_i}+x^*_i}\dot{e}_{x_i}\\
&=-\sum_i x^*_i(\widehat{p_i}-\sum_j (e_{x_j}+x^*_j) \widehat{p_j})\\
&=-\sum_i x^*_i\widehat{p_i}+\sum_i x^*_i(\sum_j e_{x_j}\widehat{p_j}+\sum_jx^*_j \widehat{p_j})\\
&=-\sum_i x^*_i\widehat{p_i}+\sum_j e_{x_j}\widehat{p_j}+\sum_jx^*_j\widehat{p_j}\\
&=\sum_j e_{x_j}\widehat{p_j}\\
&=\sum_j e_{x_j} \int_0^t p_j(\tau) d \tau.
\end{align*}
This  implies that second order replicator dynamic is negative imaginary.
\end{proof}

It is straightforward to show  that the linearization of the replicator dynamics around an equilibrium point, which is given as
\begin{align*}
\delta \dot{x}&= A \delta x+\beta\delta \widehat{p} \\%\label{eq:RDla}\\
\delta \dot{\widehat{p} }&= \delta p,%\label{eq:RDlb}
\end{align*}
satisfies the negative imaginary property. 
Here, $A=\begin{bmatrix}
-x_1^* &-x_1^*&\cdots& -x_1^*\\
-x_2^* &-x_2^*&\cdots& -x_2^*\\
\vdots&\vdots&\vdots&\vdots\\
-x_n^* &-x_n^*& \cdots& -x_n^*
\end{bmatrix}$, $\beta=I-x^*x^{*T}$ and $x^*$ is any point in the simplex. The reduced system can be obtained using  the transformation $\delta x=\delta x^*+N \delta w$,  $\delta p=\delta p^*+N\delta q$ and $\delta \widehat{p}=N\xi$ as follows:
\begin{align*}
\delta \dot{w}&= N^TA N\delta w+N^T\beta N \xi\\%\label{eq:RDla}\\
 \dot{\xi }&= \delta q.%\label{eq:RDlb}
\end{align*}
Now, consider the case where $n=3$ and $x_i^*=\frac{1}{3}$,
\begin{subequations}\label{eq:RDLR3}
\begin{align} 
\begin{bmatrix}\delta \dot{w}\\
 \dot{\xi }
\end{bmatrix} &= \begin{bmatrix} 0&0&1&0\\
0&0&0&1\\
0&0&0&0\\
0&0&0&0
\end{bmatrix} \begin{bmatrix}\delta w\\
\xi
\end{bmatrix} +
\begin{bmatrix}0 &0\\
0&0\\1&0\\0&1 
\end{bmatrix} \delta q \\
\delta y&= \begin{bmatrix}1&0&0&0\\0&1&0&0
\end{bmatrix} \begin{bmatrix}\delta w\\
\xi
\end{bmatrix}.
\end{align}
\end{subequations}
The system \eqref{eq:RDLR3}  is a double integrator, which is a negative imaginary system.

 Proposition \ref{HoRDNI} shows that the second order  replicator dynamics defined in \eqref{eq:RD2order} satisfies the negative imaginary property. This implies that the feedback interconnection between the second order  replicator dynamics defined in \eqref{eq:RD2order} and any strictly negative imaginary game should lead to a stable behaviour \cite{Mabrok2012}.

For instance, consider the following second order negative imaginary game; 
\begin{align}
 \dot{ z}&=Az+B\; u \notag\\
   P&= C  z +D\; u, \label{eq:ynigame}
\end{align}
where, 
\begin{align*}
A&=\begin{bmatrix}
-0.9&  0\\
    0  &  -1.2
\end{bmatrix},B = \begin{bmatrix}
1&0\\0&1
\end{bmatrix},\\C&= \begin{bmatrix}
 1 &  0\\0&1
\end{bmatrix}  \text{ and } D =  \begin{bmatrix}
 -3 &  0\\0&-3
\end{bmatrix}.
\end{align*}

The feedback interconnection between the second order  replicator dynamics   \eqref{eq:RD2order} and the second order strictly negative imaginary game \eqref{eq:ynigame} is demonstrated in Fig. ref{fig:RD-g}  and Fig. \ref{RD-simplex-j}.

Figs. \ref{fig:RD-g}  and \ref{RD-simplex-j} shows  the evolution of the states of the  positive feedback interconnection between the  strictly negative imaginary game \eqref{eq:ynigame} and  second  order  replicator dynamic \eqref{eq:RD2order}.

\begin{figure}[H]
%\begin{center}
\centering{\includegraphics[height=4.7 cm]{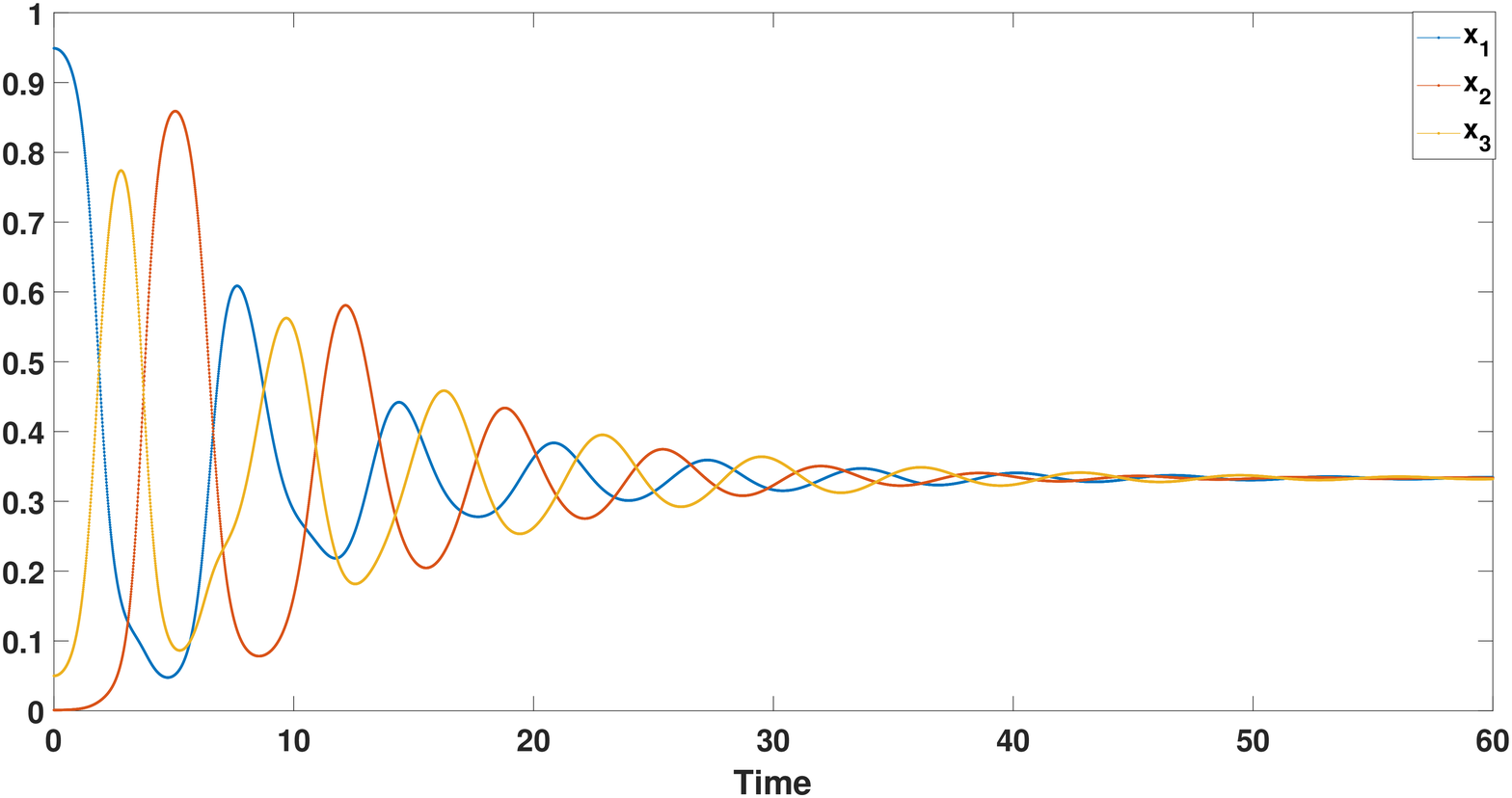}}
\caption{The evolution of the states of the strictly negative imaginary game \eqref{eq:ynigame} in feedback loop with the  second order replicator dynamics \eqref{eq:RD2order}.  }\label{fig:RD-g}.
%\end{center}
\end{figure}

\begin{figure}[H]
%\begin{center}
\centering{\includegraphics[height=8 cm]{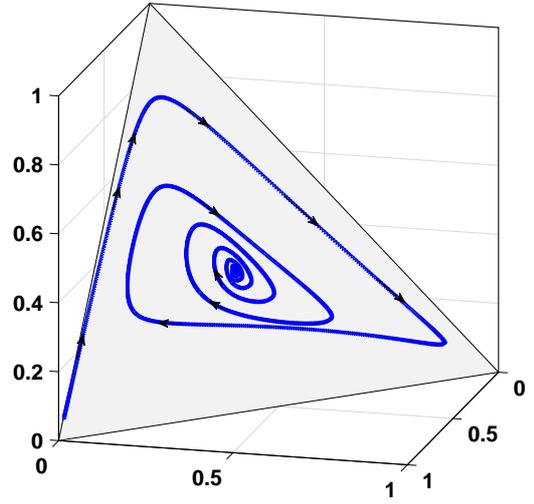}}
\caption{The evolution of the states  of the  strictly negative imaginary game \eqref{eq:ynigame}  in feedback loop with the  second order replicator dynamics \eqref{eq:RD2order} projected into the simplex.  }\label{RD-simplex-j}.
%\end{center}
\end{figure}
\subsubsection{Lead-lag second order replicator dynamics}
A more general second order replicator dynamics can be introduced using derivative action combined with low pass filter as a prepossessing to the payoffs before  the replicator dynamics as shown in Fig. \ref{fig:Ption:RD}. The Lead-lag second order replicator dynamics can be defined using an auxiliary state $\widehat{p}$. The dynamics is given as follows;  

\begin{subequations}\label{eq:RD2ordernew}
\begin{align}
\dot{x_i}&=x_i(y_i-\sum_j x_j y_j)\label{eq:RD2orderanew}\\
\dot{\widehat{p}_i}&=\frac{-1}{\beta}\widehat{p} + p_i\label{eq:RD2orderbnew}\\
y_i &= (\frac{1}{\beta}-\frac{\alpha}{\beta^2})  \widehat{p} + \frac{\alpha}{\beta} p_i.
\end{align}
\end{subequations}
Here, $\alpha$ and $\beta$ are positive constants.  

\begin{figure}[H]
%\begin{center}
\centering{\includegraphics[height=4 cm]{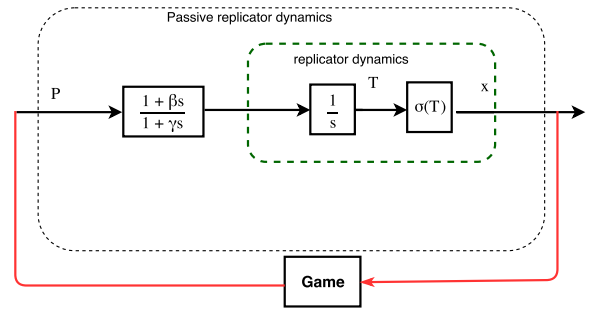}}
\caption{Passive replicator dynamics.  }\label{fig:Ption:RD}.
%\end{center}
\end{figure} 

The passivity propriety of  second order replicator dynamics given in \eqref{eq:RD2ordernew} depends on the ratio  between the constants $\alpha$ and $\beta$. The following preposition discuss the passivity property of the replicator dynamics given in \eqref{eq:RD2ordernew}. 

\begin{proposition}\label{Pro:LedSORD}
The second order replicator dynamics given in \eqref{eq:RD2ordernew} is negative imaginary for all positive constant  $\alpha$ and $\beta$ and  passive if and only if $\frac{\alpha}{\beta}>1$. 
\end{proposition}

\begin{proof}
The dynamical part of the second order replicator dynamic  given in \eqref{eq:RD2ordernew} can be expressed as a serial connection between an integrator, $\frac{1}{s}$, and the  following transfer function 
\begin{align*}
  \frac{1+\alpha \; s}{1+\beta \; s}. 
\end{align*}
In other words, the dynamical part of second order replicator dynamic  given in \eqref{eq:RD2ordernew} in frequency domain  is given as follows; 
\begin{align}\label{eq:RD_tf}
  G_{RD}(j\omega)=\frac{1+\alpha \; j\omega}{j\omega\;(1+\beta \; j\omega) }. 
\end{align}
The real and imaginary part of the transfer function   \eqref{eq:RD_tf} is given as follows; 

\begin{align}\label{eq:RD_re}
  \Re(G_{RD})=\frac{\beta (\alpha/\beta - 1)}{\beta^2 \omega^2 +1}. 
\end{align}

\begin{align}\label{eq:RD_im}
  \Im(G_{RD})=\frac{-(\alpha \beta \omega^2 + 1)}{\beta^2 \omega^3 +\omega }. 
\end{align}
It is clear from \eqref{eq:RD_re} that the real part of $G_{RD}(j\omega)$    is positive if and only if $\frac{\alpha}{\beta}>1$. Also, \eqref{eq:RD_im} shows that the imaginary part is always negative for all positive $\alpha$ and $\beta$. 

\end{proof}

Fig. \ref{RPS_RD_led_lag} demonstrates the  feedback interconnection between the standard rock paper scissors game \eqref{eq:rpsgame} with  $w=l$, i.e., lossless game  and the second order   replicator dynamics given  \eqref{eq:RD2ordernew}.
The top figure in Fig. \ref{RPS_RD_led_lag} represents the non-passive replicator dynamics where $\frac{\alpha}{\beta}<1$, where the bottom figure represents the passive replicator dynamics  with $\frac{\alpha}{\beta}>1$ .

\begin{figure}[H]
%\begin{center}
\centering{\includegraphics[height=5.3 cm]{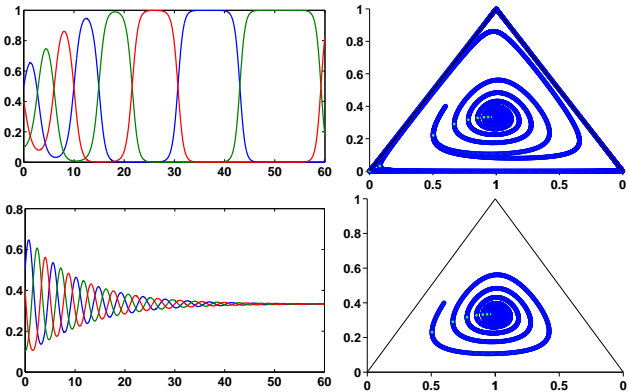}}
\caption{The evolution of the states of the strictly negative imaginary game \eqref{eq:ynigame} in feedback loop with the  second order replicator dynamics \eqref{eq:RD2order}.  }\label{RPS_RD_led_lag}.
%\end{center}
\end{figure}

\section{Passivation of non-passive dynamics }
This section  introduces  the notion of \textit{passivation} thorough a feedback or forwarded of non-passive dynamics and static gain. The term passivation in passivity theory means the process  of forcing a non-passive or lossless  dynamical system to behave as a strictly passive system using feedback and/or foreword interconnection. Similarly,

\begin{figure}[H]
%\begin{center}
\centering{\includegraphics[height=5.0 cm]{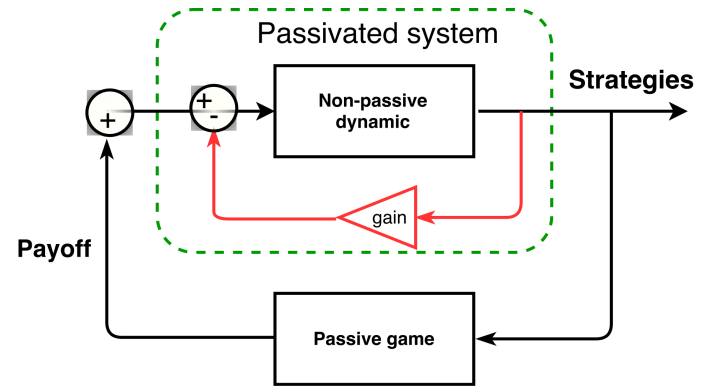}}
\caption{Passiviation of lossless or non-passive dynamics throgh static feedback gain.  Here, the static gain is chosen to passivate the dynamics in the colored green box.  }\label{fig:passivation1}.
%\end{center}
\end{figure}

 In the case of replicator dynamics, any positive gain in negative feedback interconnection will change the replicator dynamics for lossless dynamics to strictly passive dynamics in Fig. \ref{fig:RD-p}.

\begin{figure}[H]
%\begin{center}
\centering{\includegraphics[height=4 cm]{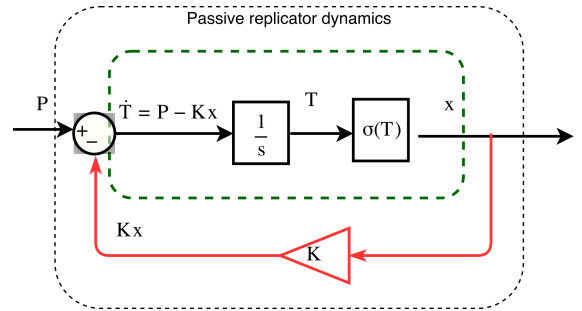}}
\caption{Passive replicator dynamics.  }\label{fig:RD-p}.
%\end{center}
\end{figure} 

The passive first order replicator dynamics is then given as follows; 
\begin{align}\label{eq:PRD}
\dot{x_i}&=x_i(\dot{T_i}-\sum_j x_j \dot{T_j}) \notag\\
\dot{T_i}&=P_i-K_i x_i.
\end{align}

Similarly, as in Proposition \ref{Pro:LedSORD}, we can show that the modified replicator dynamics \eqref{eq:PRD} a is strictly passive.

\section{Concluding Remarks}

In this paper, passivity  analysis  for  different versions  of the replicator dynamics was conducted. The standard first order replicator dynamics is shown to be passive lossless. This implies that the replicator dynamics will lead to a stable equilibrium with all strictly passive games. Similarly, the second order replicator dynamics is shown to satisfy the negative imaginary property, which implies that  the second order replicator dynamics leads to a stable equilibrium with all strictly negative imaginary games.

\bibliographystyle{IEEEtran}
%\bibliography{biblio}

\begin{thebibliography}{10}
	\providecommand{\url}[1]{#1}
	\csname url@samestyle\endcsname
	\providecommand{\newblock}{\relax}
	\providecommand{\bibinfo}[2]{#2}
	\providecommand{\BIBentrySTDinterwordspacing}{\spaceskip=0pt\relax}
	\providecommand{\BIBentryALTinterwordstretchfactor}{4}
	\providecommand{\BIBentryALTinterwordspacing}{\spaceskip=\fontdimen2\font plus
		\BIBentryALTinterwordstretchfactor\fontdimen3\font minus
		\fontdimen4\font\relax}
	\providecommand{\BIBforeignlanguage}[2]{{%
			\expandafter\ifx\csname l@#1\endcsname\relax
			\typeout{** WARNING: IEEEtran.bst: No hyphenation pattern has been}%
			\typeout{** loaded for the language `#1'. Using the pattern for}%
			\typeout{** the default language instead.}%
			\else
			\language=\csname l@#1\endcsname
			\fi
			#2}}
	\providecommand{\BIBdecl}{\relax}
	\BIBdecl
	
	\bibitem{Sandholm2010}
	W.~H. Sandholm, \emph{Population Games and Evolutionary Dynamics}.\hskip 1em
	plus 0.5em minus 0.4em\relax MIT Press, 2010.
	
	\bibitem{Hofbauer1998}
	J.~Hofbauer and K.~Sigmund, \emph{Evolutionary Games and Population
		Dynamics}.\hskip 1em plus 0.5em minus 0.4em\relax Cambridge, UK: Cambridge
	University Press, 1998.
	
	\bibitem{Young1998}
	H.~P. Young, \emph{Individual Strategy and Social Structure}.\hskip 1em plus
	0.5em minus 0.4em\relax Princeton, NJ: Princeton University Press, 1998.
	
	\bibitem{Smith1982}
	J.~M. Smith, \emph{Evolution and the Theory of Games}.\hskip 1em plus 0.5em
	minus 0.4em\relax Cambridge University Press, 1982.
	
	\bibitem{Marden2015}
	J.~Marden and J.~S. Shamma, ``Game theory and distributed control,'' in
	\emph{Handbook of Game Theory}, H.~P. Young and S.~Zamir, Eds.\hskip 1em plus
	0.5em minus 0.4em\relax North-Holland, 2015, vol.~4, pp. 861--899.
	
	\bibitem{Fudenberg1998}
	D.~Fudenberg and D.~Levine, \emph{The Theory of Learning in Games}.\hskip 1em
	plus 0.5em minus 0.4em\relax Cambridge, MA: MIT Press, 1998.
	
	\bibitem{Hart2005}
	S.~Hart, ``Adaptive heuristics,'' \emph{Econometrica}, vol.~73, no.~5, pp.
	1401--1430, 2005.
	
	\bibitem{Young2005}
	H.~P. Young, \emph{Strategic Learning and its Limits}.\hskip 1em plus 0.5em
	minus 0.4em\relax Oxford University Press, 2005.
	
	\bibitem{Hofbauer2003}
	J.~Hofbauer and K.~Sigmund, ``Evolutionary game dynamics,'' \emph{Bulletin of
		the American Mathematical Society}, vol.~40, no.~4, pp. 479--519, 2003.
	
	\bibitem{Hofbauer2009}
	J.~Hofbauer and W.~H. Sandholm, ``Stable games and their dynamics,''
	\emph{Journal of Economic Theory}, vol. 144, no.~4, pp. 1665--1693, 2009.
	
	\bibitem{Hart2003}
	S.~Hart and A.~Mas-Colell, ``Uncoupled dynamics do not lead to {N}ash
	equilibrium,'' \emph{American Economic Review}, vol.~93, no.~5, pp.
	1830--1836, 2003.
	
	\bibitem{taylor1978evolutionary}
	P.~D. Taylor and L.~B. Jonker, ``Evolutionary stable strategies and game
	dynamics,'' \emph{Mathematical biosciences}, vol.~40, no. 1-2, pp. 145--156,
	1978.
	
	\bibitem{Schuster1983}
	P.~Schuster and K.~Sigmund, ``Replicator dynamics,'' \emph{Journal of
		Theoretical Biology}, vol. 100, no.~3, pp. 533--538, 1983.
	
	\bibitem{cressman2014replicator}
	R.~Cressman and Y.~Tao, ``The replicator equation and other game dynamics,''
	\emph{Proceedings of the National Academy of Sciences}, vol. 111, no.
	Supplement 3, pp. 10\,810--10\,817, 2014.
	
	\bibitem{Fox2013}
	M.~J. Fox and J.~S. Shamma, ``Population games, stable games, and passivity,''
	\emph{Games}, vol.~4, no.~4, pp. 561--583, 2013.
	
	\bibitem{Sato2002}
	S.~Sato, E.~Akiyama, and J.~D. Farmer, ``Chaos in learning a simple two person
	game,'' \emph{Proceedings of the National Academy of Sciences}, vol.~99,
	no.~7, pp. 4748--4751, 2002.
	
	\bibitem{Shamma2005}
	J.~S. Shamma and G.~Arslan, ``Dynamic fictitious play, dynamic gradient play,
	and distributed convergence to nash equilibria,'' \emph{IEEE Transactions on
		Automatic Control}, vol.~50, no.~3, pp. 312--327, 2005.
	
	\bibitem{Arslan2006}
	G.~Arslan and J.~S. Shamma, ``Anticipatory learning in general evolutionary
	games,'' in \emph{45th IEEE Conference on Decision and Control}, San Diego,
	CA, December 2006, pp. 6289--6294.
	
	\bibitem{Laraki2013}
	R.~Laraki and P.~Mertikopoulos, ``Higer order game dynamics,'' \emph{Journal of
		Economic Theory}, vol. 148, no.~6, pp. 2666--2695, 2013.
	
	\bibitem{Lozano2013}
	R.~Lozano, B.~Brogliato, O.~Egeland, and B.~Maschke, \emph{Dissipative Systems
		Analysis and Control: Theory and Applications}.\hskip 1em plus 0.5em minus
	0.4em\relax SpringerScience \& Business Media, 2013.
	
	\bibitem{Willems1972}
	J.~C. Willems, ``Dissipative dynamical systems part {I}: {G}eneral theory,''
	\emph{Archive for Rational Mechanics and Analysis}, vol.~45, no.~5, pp.
	321--351, 1972.
	
	\bibitem{VanderSchaft2012}
	A.~V. der Schaft, \emph{L2-Gain and Passivity Tecniques in Nonlinear
		Control}.\hskip 1em plus 0.5em minus 0.4em\relax Springer Science \& Business
	Media, 2012.
	
	\bibitem{lanzon2008}
	A.~Lanzon and I.~R. Petersen, ``Stability robustness of a feedback
	interconnection of systems with negative imaginary frequency response,''
	\emph{IEEE Transactions on Automatic Control}, vol.~53, no.~4, pp.
	1042--1046, 2008.
	
	\bibitem{petersen2010}
	I.~R. Petersen and A.~Lanzon, ``Feedback control of negative imaginary
	systems,'' \emph{IEEE Control System Magazine}, vol.~30, no.~5, pp. 54--72,
	2010.
	
	\bibitem{xiong21010jor}
	J.~Xiong, I.~R. Petersen, and A.~Lanzon, ``A negative imaginary lemma and the
	stability of interconnections of linear negative imaginary systems,''
	\emph{IEEE Transactions on Automatic Control}, vol.~55, no.~10, pp.
	2342--2347, 2010.
	
	\bibitem{Mabrok2014c}
	M.~Mabrok, M.~Haggag, I.~Petersen, and A.~Lanzon, ``A subspace system
	identification algorithm guaranteeing the negative imaginary property,'' in
	\emph{Decision and Control (CDC), 2014 IEEE 53rd Annual Conference on}.\hskip
	1em plus 0.5em minus 0.4em\relax IEEE, 2014, pp. 3180--3185.
	
	\bibitem{Mabrok2015}
	M.~A. Mabrok, M.~Efatmaneshnik, and M.~Ryan, ``Including non-functional
	requirements in the axiomatic design process,'' in \emph{Systems Conference
		(SysCon), 2015 9th Annual IEEE International}.\hskip 1em plus 0.5em minus
	0.4em\relax IEEE, 2015, pp. 54--60.
	
	\bibitem{Ghallab2018}
	A.~G.~G. M.~A.~M. and Ian~R.~Petersen, ``Extending negative imaginary systems
	theory to nonlinear systems,'' 2018.
	
	\bibitem{Mabrok2013}
	M.~Mabrok, A.~Kallapur, I.~Petersen, and A.~Lanzon, ``Generalizing negative
	imaginary systems theory to include free body dynamics: Control of highly
	resonant structures with free body motion,'' \emph{Automatic Control, IEEE
		Transactions on}, vol.~59, no.~10, pp. 2692--2707, Oct 2014.
	
	\bibitem{Mabrok2012}
	M.~A. Mabrok, A.~G. Kallapur, I.~R. Petersen, and A.~Lanzon, ``A stability
	result on the feedback interconnection of negative imaginary systems with
	poles at the origin,'' in \emph{Proceedings of the 2012 Australian Control
		Conference, Sydney, Australia}, 2012.
	
	\bibitem{Glover1996}
	K.~Zhou, J.~C. Doyle, and K.~Glover, \emph{Robust and Optimal Control.}\hskip
	1em plus 0.5em minus 0.4em\relax Upper Saddle River, NJ: Prentice-Hall, Inc.,
	1996.
	
	\bibitem{Arrow1963}
	K.~J. Arrow, \emph{Social choice and individual values}.\hskip 1em plus 0.5em
	minus 0.4em\relax Yale university press, 1963, no.~12.
	
	\bibitem{ohtsuki2006replicator}
	H.~Ohtsuki and M.~A. Nowak, ``The replicator equation on graphs,''
	\emph{Journal of theoretical biology}, vol. 243, no.~1, pp. 86--97, 2006.
	
	\bibitem{bravo2015distributed}
	G.~D.~O. Bravo, ``Distributed methods for resource allocation: a passivity
	based approach,'' Ph.D. dissertation, Ecole des Mines de Nantes, 2015.
	
	\bibitem{pantoja2011dispatch}
	A.~Pantoja, N.~Quijano, and K.~M. Passino, ``Dispatch of distributed generators
	using a local replicator equation,'' in \emph{Decision and Control and
		European Control Conference (CDC-ECC), 2011 50th IEEE Conference on}.\hskip
	1em plus 0.5em minus 0.4em\relax IEEE, 2011, pp. 7494--7499.
	
	\bibitem{hilbe2011local}
	C.~Hilbe, ``Local replicator dynamics: a simple link between deterministic and
	stochastic models of evolutionary game theory,'' \emph{Bulletin of
		mathematical biology}, vol.~73, no.~9, pp. 2068--2087, 2011.
	
	\bibitem{laraki2013higher}
	R.~Laraki and P.~Mertikopoulos, ``Higher order game dynamics,'' \emph{Journal
		of Economic Theory}, vol. 148, no.~6, pp. 2666--2695, 2013.
	
	\bibitem{mabrok2016a}
	M.~Mabrok and J.~S. Shamma, ``Passivity analysis of higher order evolutionary
	dynamics and population games,'' in \emph{Decision and Control (CDC), 2016
		IEEE 55th Conference on}.\hskip 1em plus 0.5em minus 0.4em\relax IEEE, 2016,
	pp. 6129--6134.
	
\end{thebibliography}

\begin{biography}[{\includegraphics[width=2in,height=1.225in,clip,keepaspectratio]{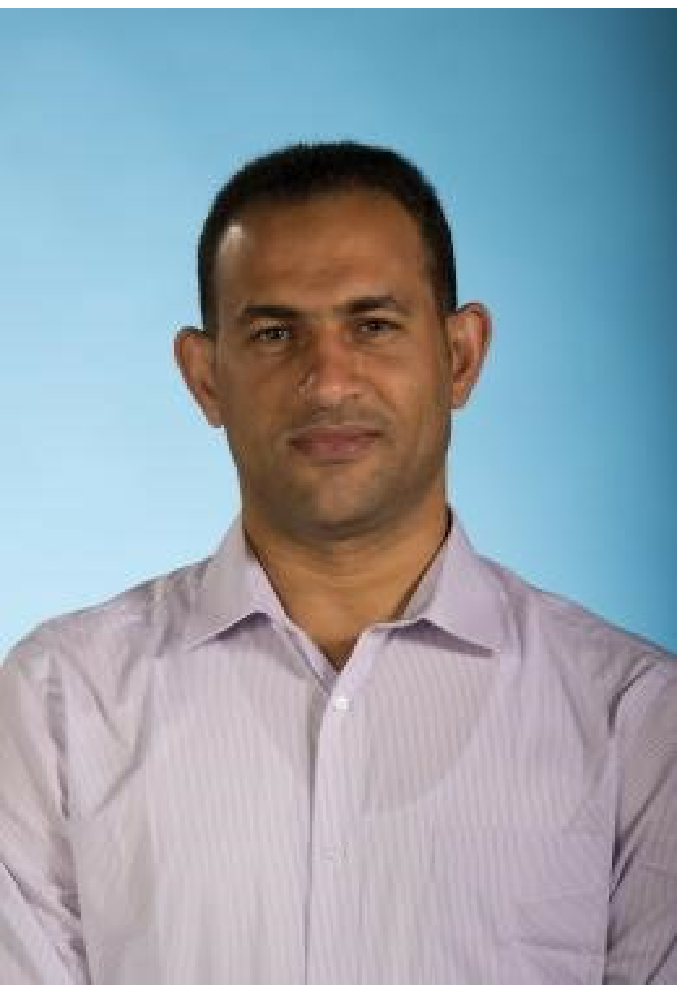}}]{Dr. Mohamed A. Mabrok}
  is an Assistant Professor with the Faculty of Engineering at the Australian College of Kuwait. In 2004 he obtained a   bachelor  degree applied mathematics and computer science and in 2009 he obtained a masters degree in applied mathematics and quantum physics. In 2013, Dr. Mabrok obtained his  Ph.D., in electrical engineering in  the field of robust control theory   from the University of New South Wales, Australia, with Professor Ian Petersen. 
  
  After finishing his Ph.D., Dr. Mabrok  worked as a postdoctoral fellow at the the UNSW at the Australian Defence Force Academy for two years. In 2015, he joined  professor Jeff Shamma' group at KAUST as a postdoctoral fellow for two years. In 2017, he joined the Australian College of Kuwait as an assistant professor. 

Dr. Mabrok's main research interests and experience   are in robust control theory, game theoretic learning, quantum control theory, and Human in the loop control. 
\end{biography}
\end{document}